\documentclass[a4paper, 12pt]{amsart}

\usepackage{amsmath,amssymb,enumitem,verbatim,stmaryrd,xcolor,microtype,graphicx,aliascnt}

\usepackage[T1]{fontenc}
\usepackage[utf8]{inputenc}
\usepackage[english]{babel}
\usepackage[top=3.4cm,bottom=3.4cm,left=3.2cm,right=3.2cm]{geometry}
\usepackage[bookmarksdepth=3,linktoc=page,colorlinks,linkcolor={red!80!black},citecolor={red!80!black},urlcolor={blue!80!black}]{hyperref}

\usepackage{tikz}\usetikzlibrary{matrix,arrows,decorations.markings}
\usepackage{tikz-cd}

\usepackage[mathcal]{euscript}                               % use the euscript caligraphic font with \mathscr{...}
\usepackage{mathptmx}                                        % times font
\usepackage{etoolbox}\makeatletter\patchcmd{\@startsection}{\@afterindenttrue}{\@afterindentfalse}{}{}\makeatother    %omit indentation of the first paragraph of a section
\patchcmd{\section}{\scshape}{\bfseries}{}{}\makeatletter\renewcommand{\@secnumfont}{\bfseries}\makeatother           %boldface section and subsection titles (no caption), including numbers
\usepackage[backgroundcolor=yellow,linecolor=yellow,textsize=footnotesize]{todonotes} \makeatletter \providecommand \@dotsep{5} \def\listtodoname{List of Todos} \def\listoftodos{\@starttoc{tdo}\listtodoname} \makeatother %\todo{} for margin notes, supress in pdf with option [disable]

\theoremstyle{plain}
\newtheorem{thm}{Theorem}[section] % provides command \autoref{}, which produces citations like ``Theorem 1.1''.
\newaliascnt{lemma}{thm}\newtheorem{lemma}[lemma]{Lemma}\aliascntresetthe{lemma}
\newaliascnt{cor}{thm}\newtheorem{cor}[cor]{Corollary}\aliascntresetthe{cor}
\newaliascnt{prop}{thm}\newtheorem{prop}[prop]{Proposition}\aliascntresetthe{prop}

\newtheorem{thmA}{Theorem} %alphabetic theorem counter: Theorem A, Theorem B, ...
\newtheorem{propA}[thmA]{Theorem}

\newtheorem*{thm*}{Theorem}
\newtheorem*{lem*}{Lemma}
\newtheorem*{cor*}{Corollary}

\theoremstyle{definition}
\newaliascnt{df}{thm}\newtheorem{df}[df]{Definition}\aliascntresetthe{df}
\newaliascnt{rem}{thm}\newtheorem{rem}[rem]{Remark}\aliascntresetthe{rem}
\newaliascnt{ex}{thm}\aliascntresetthe{ex}

\newtheorem*{df*}{Definition}
\newtheorem*{ex*}{Example}
\newtheorem*{rem*}{Remark}

\usepackage{etoolbox}
\makeatletter
\patchcmd{\@startsection}{\@afterindenttrue}{\@afterindentfalse}{}{}             %omit indentation of the first paragraph of a section
\patchcmd{\part}{\bfseries}{\bfseries\LARGE}{}{}
\patchcmd{\section}{\scshape}{\bfseries}{}{}\renewcommand{\@secnumfont}{\bfseries} %boldface no smallcaps section and subsection titles with numbers
\patchcmd{\@settitle}{\uppercasenonmath\@title}{\large}{}{}
\patchcmd{\@setauthors}{\MakeUppercase}{}{}{}
\addto{\captionsenglish}{} %boldface no smallcaps Abstract
\addto{\captionsenglish}{} %bolface no smallcaps Figure
\addto{\captionsenglish}{} %bolface no smallcaps Table
\makeatother

\usepackage{fancyhdr}

\DeclareRobustCommand{\gobblefour}[4]{}    % Command \SkipTocEntry for surpressing a section title in TOC
\DeclareSymbolFont{sfoperators}{OT1}{bch}{m}{n} \DeclareSymbolFontAlphabet{\mathsf}{sfoperators} \makeatletter\def\operator@font{\mathgroup\symsfoperators}\makeatother % different font for math operators
\DeclareSymbolFont{cmletters}{OML}{cmm}{m}{it}              
\DeclareSymbolFont{cmsymbols}{OMS}{cmsy}{m}{n}
\DeclareSymbolFont{cmlargesymbols}{OMX}{cmex}{m}{n}
\DeclareMathSymbol{\myjmath}{\mathord}{cmletters}{"7C}     \let\jmath\myjmath %Defining the missing commands: \jmath, \amalg and \coprod
\DeclareMathSymbol{\myamalg}{\mathbin}{cmsymbols}{"71}     
\DeclareMathSymbol{\mycoprod}{\mathop}{cmlargesymbols}{"60}
\DeclareMathSymbol{\myalpha}{\mathord}{cmletters}{"0B}     \let\alpha\myalpha %Greek letters from Computer Modern
\DeclareMathSymbol{\mybeta}{\mathord}{cmletters}{"0C}      \let\beta\mybeta
\DeclareMathSymbol{\mygamma}{\mathord}{cmletters}{"0D}     \let\gamma\mygamma
\DeclareMathSymbol{\mydelta}{\mathord}{cmletters}{"0E}     \let\delta\mydelta
\DeclareMathSymbol{\myepsilon}{\mathord}{cmletters}{"0F}   \let\epsilon\myepsilon
\DeclareMathSymbol{\myzeta}{\mathord}{cmletters}{"10}      \let\zeta\myzeta
\DeclareMathSymbol{\myeta}{\mathord}{cmletters}{"11}       \let\eta\myeta
\DeclareMathSymbol{\mytheta}{\mathord}{cmletters}{"12}     \let\theta\mytheta
\DeclareMathSymbol{\myiota}{\mathord}{cmletters}{"13}      \let\iota\myiota
\DeclareMathSymbol{\mykappa}{\mathord}{cmletters}{"14}     \let\kappa\mykappa
\DeclareMathSymbol{\mylambda}{\mathord}{cmletters}{"15}    \let\lambda\mylambda
\DeclareMathSymbol{\mymu}{\mathord}{cmletters}{"16}        \let\mu\mymu
\DeclareMathSymbol{\mynu}{\mathord}{cmletters}{"17}        \let\nu\mynu
\DeclareMathSymbol{\myxi}{\mathord}{cmletters}{"18}        \let\xi\myxi
\DeclareMathSymbol{\mypi}{\mathord}{cmletters}{"19}        \let\pi\mypi
\DeclareMathSymbol{\myrho}{\mathord}{cmletters}{"1A}       \let\rho\myrho
\DeclareMathSymbol{\mysigma}{\mathord}{cmletters}{"1B}     \let\sigma\mysigma
\DeclareMathSymbol{\mytau}{\mathord}{cmletters}{"1C}       \let\tau\mytau
\DeclareMathSymbol{\myupsilon}{\mathord}{cmletters}{"1D}   \let\upsilon\myupsilon
\DeclareMathSymbol{\myphi}{\mathord}{cmletters}{"1E}       \let\phi\myphi
\DeclareMathSymbol{\mychi}{\mathord}{cmletters}{"1F}       \let\chi\mychi
\DeclareMathSymbol{\mypsi}{\mathord}{cmletters}{"20}       \let\psi\mypsi
\DeclareMathSymbol{\myomega}{\mathord}{cmletters}{"21}     \let\omega\myomega
\DeclareMathSymbol{\myvarepsilon}{\mathord}{cmletters}{"22}\let\varepsilon\myvarepsilon
\DeclareMathSymbol{\myvartheta}{\mathord}{cmletters}{"23}  \let\vartheta\myvartheta
\DeclareMathSymbol{\myvarpi}{\mathord}{cmletters}{"24}     \let\varpi\myvarpi
\DeclareMathSymbol{\myvarrho}{\mathord}{cmletters}{"25}    \let\varrho\myvarrho
\DeclareMathSymbol{\myvarsigma}{\mathord}{cmletters}{"26}  \let\varsigma\myvarsigma
\DeclareMathSymbol{\myvarphi}{\mathord}{cmletters}{"27}    \let\varphi\myvarphi

\DeclareMathOperator{\Hom}{Hom}

\DeclareMathOperator{\im}{im}

\DeclareMathOperator{\OBlpr}{{OBlpr}}

\DeclareMathOperator{\Rings}{{Rings}}
\DeclareMathOperator{\SRings}{{SRings}}
\DeclareMathOperator{\Mon}{{Mon}}

\newcommand\F{{\mathbb F}}

\renewcommand\int{\textup{int}}

\newcommand\hull{\textup{hull}}

\newcommand\idem{\textup{idem}}

\newcommand\mon{\textup{mon}}

\newcommand\IdempSFields{\textup{IdempSFields}}
\newcommand\Bil{\textup{Bil}}
\newcommand\Mod{\textup{Mod}}
\newcommand{\quot}[2]{#1\!\sslash\!#2}

\renewcommand\geq{\geqslant}
\renewcommand\leq{\leqslant}

\newcommand{\bigbpgenquot}[2]{#1\big\slash\!\!\!\big\slash\!\big\langle\,{#2}\,\big\rangle}

\newcommand\Funpm{{\F_1^\pm}}

\title{Exterior algebras in matroid theory}

\author{Manoel Jarra}
\address{\rm Manoel Jarra, University of Groningen, the Netherlands, and IMPA, Rio de Janeiro, Brazil}
\email{{m.zanoelo.jarra@rug.nl}}

\begin{document}

\begin{abstract}
Ordered blueprints are algebraic objects that generalize monoids and ordered semirings, and $\mathbb{F}_1^{\pm}$-algebras are ordered blueprints that have an element $\epsilon$ that acts as $-1$. In this work we introduce an analogue of the exterior algebra for $\mathbb{F}_1^{\pm}$-algebras that provides a new cryptomorphism for matroids. We also show how to recover the usual exterior algebra if the $\mathbb{F}_1^{\pm}$-algebra comes from a ring, and the Giansiracusa Grassmann algebra if the $\mathbb{F}_1^{\pm}$-algebra comes from an idempotent semifield.
\end{abstract}

\maketitle

%{\ \vspace{-315pt}\\ \flushright\tiny\bf\ \today\\ {}\ \\}\vspace{280pt}

\begin{small} \tableofcontents \end{small}

%%%%%%%%%%%%%%%%%%%%%%%%%%%%%%%%%%%%%%%%%%%%%%%%%%%%%%%%%%%%%%%%%%%%%%%%%%%%%%%%%%%%%%%%%%%%%%%%%%%%%%%%%%%%%%%%%%%%%%%%%%%%%%%%%%%%%%%%%%%%%%%%%%%%%%%%%%%%%%%%%%%%%%%%%%%
%%%%%%%%%%%%%%%%%%%%%%%%%%%%%%%%%%%%%%%%%%%%%%%%%%%%%%%%%%%%%%%%%%%%%%%%%%%%%%%%%%%%%%%%%%%%%%%%%%%%%%%%%%%%%%%%%%%%%%%%%%%%%%%%%%%%%%%%%%%%%%%%%%%%%%%%%%%%%%%%%%%%%%%%%%%

\section*{Introduction}
\label{introduction}               

It is a classical theme that $d$-dimensional linear subspaces of the vector space $K^n$ over
a field $K$ correspond to certain elements of the exterior algebra $\Lambda K^{n \choose d}$, which are well-defined
up to scalar multiples in $K^\times$.

The combinatorial counterpart of such linear subspaces are matroids. Baker and Bowler streamline in \cite{Baker-Bowler19} this analogy in a broad sense by the theory of matroids with coefficients in so-called tracts. Fields, semifields and more generally hyperfields, can be seen as examples of tracts. 

Jeffrey and Noah Giansiracusa introduce in \cite{Giansiracusa18} an exterior algebra for idempotent semifields $S$ and exhibit a ‘cryptomorphic’ description of $S$-matroids in terms of the exterior algebra, in a formal analogy to the description of $K$-matroids, or linear subspaces of $K^n$, in the case of a field $K$.

Somewhat puzzling, however, is that Giansiracusas' definition of the exterior algebra for idempotent semifields makes explicit use of the idempotency in the sense that for a free module with basis $\{e_1, \dotsc, e_n\}$, one has that $e_i\otimes e_i = 0$, for $i=1, \dotsc, n$, are the only defining relations, in contrast to the larger set of relations for fields.

In this paper, we give a unified approach to both the classical theory over fields and Giansiracusas’ theory for idempotent semifields, which is based on Lorscheid’s theory of ordered blueprints (cf. \cite{Lorscheid18}). Both fields and idempotent semifields can be realized as ordered blueprints in terms of faithful functors:
\[
(-)^\mon: \Rings \longrightarrow \OBlpr \qquad \text{and} \qquad (-)^\mon: \IdempSFields \longrightarrow \OBlpr
\]
which have respective left inverses
\[
(-)^\hull: \OBlpr \longrightarrow \Rings \qquad \text{and} \qquad (-)^\idem: \OBlpr \longrightarrow \IdempSFields.
\]

We define an exterior algebra for ordered blueprints and show that it recovers both classical exterior algebra over rings as well as Giansiracusas’ exterior algebra over idempotent semifields in terms of these functors. Moreover, we give a cryptomorphic description of matroids over $\F^\pm_1$-algebras, as introduced by Baker and Lorscheid in \cite{Baker-Lorscheid21}, by using elements of the exterior algebra, which recovers the classical viewpoint on linear subspaces of $K^n$ and Giansiracusas’ interpretation of matroids over idempotent semifields. 

\subsection*{Description of results}

Let $B$ be an $\Funpm$-algebra and $n$ be an integer. The exterior algebra $\Lambda B^n=\underset{i \geq 0}{\bigoplus} \Lambda^i B^n$ of $B^n$ is a $B$-module whose underlying semigroup is a (typically non-commutative) $B^+$-algebra. This exterior algebra bears properties analogous to the classical exterior algebra.

\begin{propA}
    \label{propositionA}
 There are $B$-linear isomorphisms of ordered blue $B$-modules
 \[
  B \ \stackrel\sim\longrightarrow \ \Lambda^0  B^n , \qquad B^n \ \stackrel\sim\longrightarrow \ \Lambda^1 B^n , \qquad 0  \ \stackrel\sim\longrightarrow \ \Lambda^k B^n \enspace \text{for} \enspace k\geq n+1,
 \]
 and $(\Lambda B^n)^+$ is generated by $(\Lambda^1 B^n)^+$ as a $B^+$-algebra.
\end{propA}

This unifies and generalizes the classical and Giansircusas' exterior algebra in the following sense:

\begin{thmA}\hfill
\label{teoB}
 \begin{enumerate}
  \item Let $R$ be a ring and $B=R^\mon$. Then $\Lambda R^n$ is canonically isomorphic to $(\Lambda B^n)^{\hull,+}$ as an $R$-algebra.
  \item Let $S$ be an idempotent semiring and $B=S^\mon$. Then the Giansiracusa exterior algebra $\Lambda S^n$ is  canonically isomorphic to $(\Lambda B^n)^{\idem,+}$ as an $S$-algebra.
 \end{enumerate}
\end{thmA}

Let $[n]=\{1,\dotsc,n\}$ and $\Gamma=\binom{[n]}{d}$ be the family of $d$-subsets of $[n]$. We define a $B$-matroid of rank $d$ on $[n]$ as the $B^\times$-class $[\nu]$ of an element $\nu=(\nu_I)_{I\in \Gamma}$ of $\Lambda^d B^\Gamma$ that satisfies a certain system of relations (see \ref{df B-matroid}) and such that $\nu_I\in B^\times$ for some $I\in\Gamma$. This recovers the aforementioned concepts of matroids in the following sense:

\begin{thmA}
 \label{teoC}
 Consider $0\leq d\leq n$.
 \begin{enumerate}
  \item Let $K$ be a field and $B=K^\mon$. Then the isomorphism $\Lambda K^{\binom{[n]}{d}}\simeq(\Lambda B^{\binom{[n]}{d}})^{\hull,+}$ induces a bijection between $K$-matroids (as in \cite{Baker-Bowler19}) and $B$-matroids.
  \item Let $S$ be an idempotent semifield and $B=S^\mon$. Then the isomorphism $\Lambda S^{\binom{[n]}{d}}\simeq(\Lambda B^{\binom{[n]}{d}})^{\idem,+}$ induces a bijection between tropical Pl\"ucker vectors (as in \cite{Giansiracusa18}) and Grassmann-Pl\"ucker functions with coefficients in $B$ in the sense of this text.
  \item Let $B$ be an ordered blueprint. Then there is a canonical bijection between $B$-matroids in the sense of \cite{Baker-Lorscheid21} and classes of $B$-Pl\"ucker vectors (as defined in \ref{df Plucker vectors}).
 \end{enumerate}
\end{thmA}

\subsection*{Remark}
We draw the reader's attention to the fact that the functors $(-)^\mon$, $(-)^{\hull,+}$ and $(-)^{\idem,+}$ play the same role as in Lorscheid's approach to tropicalization as a base change from a field to the tropical hyperfield in \cite{Lorscheid19}. This indicates that our theory is part of a larger picture that puts classical theory and idempotent analysis on an equal footing.

%%%%%%%%%%%%%%%%%%%%%%%%%%%%%%%%%%%%%%%%%%%%%%%%%%%%%%%%%%%%%%%%%%%%%%%%%%%%%%%%%%%%%%%%%%%%%%%%%%%%%%%%%%%%%%%%%%%%%%%%%%%%%%%%%%%%%%%%%%%%%%%%%%%%%%%%%%%%%%%%%%%%%%%%%%%

\subsection*{Acknowledgements}
The author thanks Oliver Lorscheid for useful conversations and for his help with preparing this text. The present work was carried out with the support of CNPq, National Council for Scientific and Technological Development - Brazil.

%%%%%%%%%%%%%%%%%%%%%%%%%%%%%%%%%%%%%%%%%%%%%%%%%%%%%%%%%%%%%%%%%%%%%%%%%%%%%%%%%%%%%%%%%%%%%%%%%%%%%%%%%%%%%%%%%%%%%%%%%%%%%%%%%%%%%%%%%%%%%%%%%%%%%%%%%%%%%%%%%%%%%%%%%%%
%%%%%%%%%%%%%%%%%%%%%%%%%%%%%%%%%%%%%%%%%%%%%%%%%%%%%%%%%%%%%%%%%%%%%%%%%%%%%%%%%%%%%%%%%%%%%%%%%%%%%%%%%%%%%%%%%%%%%%%%%%%%%%%%%%%%%%%%%%%%%%%%%%%%%%%%%%%%%%%%%%%%%%%%%%%

\section{Algebraic background}

%%%%%%%%%%%%%%%%%%%%%%%%%%%%%%%%%%%%%%%%%%%%%%%%%%%%%%%%%%%%%%%%%%%%%%%%%%%%%%%%%%%%%%%%%%%%%%%%%%%%%%%%%%%%%%%%%%%%%%%%%%%%%%%%%%%%%%%%%%%%%%%%%%%%%%%%%%%%%%%%%%%%%%%%%%%

In this section, we review some background around $\mathbb{F}_1^\pm$-algebras, following \cite{Lorscheid18}. 

If $\tau$ is a preorder on a set $X$, viewed as a subset of $X\times X$, we will use $x\leq_\tau y$ to denote that $(x,y)$ is in $\tau$ and $a\equiv_\tau b$ to denote that both $(a,b)$ and $(b, a)$ are in $\tau$. Note that $\equiv_\tau$ is an equivalence relation. If the context is clear, we will denote $\leq_\tau$ simply by $\leq$.

%%%%%%%%%%%%%%%%%%%%%%%%%%%%%%%%%%%%%%%%%%%%%%%%%%%%%%%%%%%%%%%%%%%%%%%%%%%%%

\subsection{Monoids and $M$-sets}

A \emph{monoid} is a unital semigroup. A monoid $M$ is called \emph{pointed} if it has an \emph{absorbing element}, \emph{i.e.}, an element $0$ such that $0. m = 0 = m.0$ for all $m$ in $M$. The neutral and the absorbing element (if they exist) are always unique. For the rest of this text, unless otherwise stated, every monoid is supposed to be commutative.

A \emph{submonoid} of a monoid $M$ is monoid $N$ that is a subset of $M$, contains $1_M$ and whose operation $\cdot_N$ is the restriction of $\cdot_M$. If $M$ is pointed and its absorbing element is in $N$, we say that $N$ is a pointed submonoid of $M$.

If $M$ and $W$ are monoids, a map $f:M \to W$ is called a \emph{morphism of monoids} if $f(1_M)=1_W$ and $f(x). f(y)=f(x y)$ for all $x, y$ in $M$. If $M$ and $W$ are pointed and $f$ carries the absorbing element of $M$ to the absorbing element of $W$, we say that $f$ is a \emph{morphism of pointed monoids}. The category of pointed monoids will be denoted by $\Mon_*$.

A preorder $\mathfrak{r}$ on $M$ is called \emph{multiplicative} (or \emph{additive}, depending on the operation of $M$) if, for all elements $m$, $n$ and $x$ in $M$, one has $mx\leq_\mathfrak{r} nx$ whenever $m\leq_\mathfrak{r} n$. A \emph{congruence} is a multiplicative preorder that is symmetric (thus an equivalence relation). If $\mathfrak{r}$ is a multiplicative preorder on $M$, the set  $\mathfrak{c}_\mathfrak{r}:=\{(m, n)\in M\times M|\, m\equiv_\mathfrak{r} n\}$ is a congruence. The quotient set $M/\mathfrak{c}_\mathfrak{r}$ is a monoid, with operation $[m].[n]=[m. n]$, and has an induced multiplicative partial order $\overline{\mathfrak{r}}:=\{([a], [b]) |\, a\leq_\mathfrak{r} b\}$.

If $M$ is a monoid, an \emph{$M$-set} is a set $X$ equipped with a map 
\[
\begin{array}{ccc}
       M\times X & \longrightarrow  & X\\
        (m, x)   & \longmapsto      & m. x
\end{array}
\]
that satisfies:
\[
\;\quad (i)\; (m. n). x=m.(n. x) \;\quad \text{and} \;\quad (ii)\; 1. x = x,
\]
for all $m, n$ in $M$ and $x$ in $X$. An \emph{$M$-subset} of $X$ is an $M$-set $Z$ such that $Z$ is a subset of $X$ and whose map $M\times Z\to Z$ is the restriction of the map $M\times X\to X$.

If $M$ is a pointed monoid with absorbing element $0$, a \emph{pointed $M$-set} is a pointed set $(X,p)$ equipped with a map $M\times X\to X$ that makes $X$ an $M$-set and satisfies: 
\[
\;\quad (iii)\; 0. x=p \;\quad \text{and} \;\quad (iv)\; m. p=p,
\]
for all $m$ in $M$ and $x$ in $X$. A \emph{pointed $M$-subset} of $(X,p)$ is a pointed $M$-set $(Z,q)$ such that $Z$ is an $M$-subset of $X$ and $p=q$.

If $\{(X_i, p_i)|\, i \in I\}$ is a family of pointed $M$-sets, its \emph{coproduct} is given by $\underset{i\in I}{\bigvee} X_i:= \bigg(\underset{i\in I}{\bigsqcup} X_i\bigg)\big/\sim$, where $\sim
:=\{(p_i, p_j)|\, i, j \in I\}$, equipped with the $M$-action
\[
\begin{array}{ccc}
  M\times \bigvee X_i & \longrightarrow & \bigvee X_i\\
  (m, [x_j]) & \longmapsto & [m. x_j]
\end{array}
\]
where $[y_j]$ denotes the class, in $\underset{i\in I}{\bigvee} X_i$, of $y_j\in X_j$.

If $X_1$, $X_2$ and $W$ are $M$-sets, a map $\varphi: X_1\times X_2 \to W$ is called \emph{$M$-bilinear} if
\[
\varphi(-, x_2): X_1 \to W \qquad \text{and} \qquad \varphi(x_1, -): X_2 \to W
\]
are morphims of $M$-sets for all $x_1$ in $X_1$ and $x_2$ in $X_2$. We denote the set of $M$-bilinear maps $X_1\times X_2 \to W$ by $\Bil_M(X_1\times X_2, W)$.

We construct the \emph{tensor product $X_1\otimes_{M} X_2$ of $X_1$ and $X_2$} as the quotient of $X_1\times X_2$ by the equivalence relation generated by $\{(m. x_1, x_2)\sim (x_1, m. x_2)|\; x_i\in X_i \text{ and } m\in M\}$, and denote the class of $(x_1,x_2)$ by $x_1\otimes x_2$. 

The map 
\[
\begin{array}{ccc}
       M\times (X_1\otimes_{M} X_2) & \longrightarrow  & X_1\otimes_{M} X_2\\
        (m, x_1\otimes x_2)   & \longmapsto      & (m. x_1)\otimes x_2=x_1\otimes (m. x_2)
\end{array}
\]
turns $X_1\otimes_{M} X_2$ into an $M$-set.

If $\varphi: X_1\times X_2 \to W$ is an $M$-bilinear map, then 
\[
\begin{array}{cccc}
       \overline{\varphi}: & X_1\otimes_{M} X_2 & \longrightarrow  & W\\
         & x_1\otimes x_2  & \longmapsto      & \varphi(x_1, x_2)
\end{array}
\]
defines a morphism of $M$-sets.

The universal property of $X_1\otimes_{M} X_2$ can be expressed as the fact that $\varphi\mapsto\overline{\varphi}$ is a bijection from $\mbox{Bil}_M(X_1\times X_2, W)$ to $\Hom_M(X_1\otimes_M X_2, W)$.

%%%%%%%%%%%%%%%%%%%%%%%%%%%%%%%%%%%%%%%%%%%%%%%%%%%%%%%%%%%%%%%%%%%%%%%%%%%%%

\subsection{Semirings and modules}

A \emph{semiring} is a triple $(S,+,\cdot)$, where $(S, +)$ is a commutative monoid with identity $0$, $(S, \cdot)$ is a (non-necessarily commutative) pointed monoid with identity $1$ and absorbing element $0$, and satisfying $(a+b). c = (a. c) + (b. c)$ and $c. (a+b)=(c. a)+(c. b)$ for all $a, b$ and $c$ in $S$. We also use $ab$ to denote $a. b$. A semiring is called \emph{commutative} if $(S, \cdot)$ is commutative. If the operations are clear, we use $S$ to denote the semiring $(S, +, \cdot)$. For the rest of this text, unless otherwise stated, every semiring is supposed to be commutative.

If $S_1$ and $S_2$ are semirings, a map $f: S_1 \to S_2$ is called a \emph{morphism of semirings} if it is, at the same time, a morphism of the underlying monoids $f: (S_1, +) \to (S_2, +)$ and $f: (S_1, \cdot) \to (S_2, \cdot)$. The category of semirings will be denoted by $\SRings$.

A preorder $\mathfrak{r}$ on a semiring $S$ is called \emph{additive and multiplicative} if $a+z\leq_\mathfrak{r} b+z$ and $az\leq_\mathfrak{r} bz$ whenever $(a, b)$ is in $\mathfrak{r}$ and $z$ is in $S$. 

An \emph{ordered semiring} is a pair $(S, \leq)$, where $S$ is a semiring and $\leq$ is an additive and multiplicative partial order on $S$. A \emph{morphism of ordered semirings} is a morphism of semirings that is order-preserving.

A \emph{congruence} on a semiring $S$ is an additive and multiplicative preorder that is symmetric (thus an equivalence relation). If $\mathfrak{c}$ is a congruence on a semiring $S$, the quotient set $S/\mathfrak{c}$ is naturally a semiring, with operations defined by $[a]+[b]:=[a+b]$ and $[a].[b]:=[ab]$.

If $\mathfrak{r}$ is an additive and multiplicative preorder on a semiring $S$, the set  $\mathfrak{c}_\mathfrak{r}:=\{(a, b)\in S\times S|\, a\equiv_\mathfrak{r} b\}$ is a congruence and the quotient semiring $S/\mathfrak{c}_\mathfrak{r}$ has an induced additive and multiplicative partial order $\overline{\mathfrak{r}}:=\{([a], [b]) |\, a\leq_\mathfrak{r} b\}$.

An \emph{$S$-module} is a pointed monoid $(Y, +)$ with identity $0$ equipped with a map
\[
\begin{array}{cccc}
       \lambda: & S\times Y & \longrightarrow  & Y\\
        &(s, y)   & \longmapsto      & s. y
\end{array}
\]
that makes $Y$ an $(S,\cdot)$-set and satisfies 
\[
(s+r). y = (s. y) + (r. y) \;\quad\text{and}\;\quad s.(y+z) = (s. y)+(s. z)
\]
for all $s,r$ in $S$ and $y,z$ in $Y$. A map $f: Y_1\to Y_2$ is called a \emph{morphism of $S$-modules} if $f(y+z) = f(y) + f(z)$ and $f(s. y) = s. f(y)$, for all $y,z$ in $Y_1$ and $s$ in $S$.

If $Y_1$, $Y_2$ and $Z$ are $S$-modules, a map $\varphi: Y_1\times Y_2 \to Z$ is called \emph{$S$-bilinear} if
\[
\varphi(-, y_2): Y_1 \to Z \qquad \text{and} \qquad \varphi(y_1, -): Y_2 \to Z
\]
are morphims of $S$-modules for all $y_1$ in $Y_1$ and $y_2$ in $Y_2$. We denote the set of $S$-bilinear maps $Y_1\times Y_2 \to Z$ by $\Bil_S(Y_1\times Y_2, Z)$.

We construct the \emph{tensor product $Y_1\otimes_{S} Y_2$} as the quotient of $S^{Y_1\times Y_2}$ by the congruence generated by $\{e_{(s. y_1, y_2)}\equiv s. e_{(y_1, y_2)}\equiv e_{(y_1, s. y_2)}|\; s\in S \text{ and } y_i\in Y_i\}$ (where $e_{(h_1, h_2)}$ is the element of $S^{Y_1\times Y_2}$ that has $1$ in the entry correponding to $(h_1, h_2)$ and $0$ on the others). We denote the class of $e_{(y_1,y_2)}$ by $y_1\otimes y_2$. 

If $\varphi: Y_1\times Y_2 \to Z$ is an $S$-bilinear map, one has a morphism of $S$-modules characterized by
\[
\begin{array}{cccc}
       \overline{\varphi}: & Y_1\otimes_{S} Y_2 & \longrightarrow  & Z\\
         & y_1\otimes y_2  & \longmapsto      & \varphi(y_1, y_2).
\end{array}
\]
The association $\varphi\mapsto\overline{\varphi}$ is a bijection from $\Bil_S(Y_1\times Y_2, Z)$ to $\Hom_S(Y_1\otimes_S Y_2, Z)$. For a number $n\geq 1$ and an $S$-module $Y$, we will denote the $n$-fold tensor product $Y\otimes\dotsc\otimes Y$ by $Y^{\otimes n}$, and define $Y^{\otimes 0}$ as $S$.

For the $S$-module $TY:=\underset{l=0}{\overset{\infty}{\bigoplus}}(Y)^{\otimes l}$, the map 
\[
(y_1\otimes\dotsc \otimes y_i, z_1\otimes\dotsc \otimes z_j) \longmapsto y_1\otimes\dotsc \otimes y_i\otimes z_1\otimes\dotsc \otimes z_j
\]
extends, by linearity, to a product of $TY$ that makes it a (typically non-commutative) $S$-algebra.

%%%%%%%%%%%%%%%%%%%%%%%%%%%%%%%%%%%%%%%%%%%%%%%%%%%%%%%%%%%%%%%%%%%%%%%%%%%%%

\subsection{Ordered blueprints} 

An \emph{ordered blueprint} is a triple $(B^\bullet, B^+, \leq)$ such that:
\begin{enumerate}
 \item $(B^+, \leq)$ is an ordered semiring;
 \item $B^\bullet$ is a pointed submonoid of $(B^+, \cdot)$;
 \item $B^\bullet$ generates $B^+$ as a semiring, \emph{i.e.}, every element of $B^+$ is a finite sum of elements in $B^\bullet$.
\end{enumerate}

We call $B^\bullet$ the \emph{underlying monoid} and $B^+$ the \emph{ambient semiring} of $B$, and use $B^\times$ to denote the set of invertible elements of $B^\bullet$. A \emph{morphism of ordered blueprints $\varphi: B \to C$} is an order-preserving morphism $\varphi: B^+ \to C^+$ of semirings that satisfies $\varphi(B^\bullet)\subseteq C^\bullet$.

If $B$ is an ordered blueprint and $\mathfrak{r}$ is an additive and multiplicative preorder on $B^+$ containing $\leq$, one has the \emph{quotient ordered blueprint} $B\quot\mathfrak{r}:=\{\overline{B^\bullet}, B^+/\mathfrak{c}_\mathfrak{r}, \overline{\mathfrak{r}}\}$, where $\mathfrak{c}_\mathfrak{r}$ and $\overline{\mathfrak{r}}$ are the congruence and the partial order induced by $\mathfrak{r}$, respectively, and $\overline{B^\bullet}$ is the multiplicative submonoid of $B^+/\mathfrak{c}_\mathfrak{r}$ whose elements are the classes of elements in $B^\bullet$.

For a subset $H$ of $B^+\times B^+$, we define the preorder generated by $H$ as
\[
  \langle H \rangle:=\bigcap\{ \text{additive and multiplicative preorders on } B^+ \text{ containing } H \text{ and }\leq\}.
\]

If $H=\{(a_i, b_i)|\, i\in I\}$, we write $\langle a_i\leq b_i|\, i\in I\rangle$ to denote $\langle H \rangle$, and $\langle a_i\equiv b_i|\, i\in I\rangle$ to denote $\langle a_i\leq b_i \mbox{ and } b_i\leq a_i|\, i\in I\rangle$.  

For an ordered blueprint $B$, let $\mathfrak{f}:=\langle a\equiv b|\, a\leq b\rangle$ and define the \emph{algebraic hull} of $B$ as the ordered blueprint $B^{\textup{hull}}:=B\big\slash\!\!\!\big\slash\mathfrak{f}$. Note that the partial order of $B^{\textup{hull}}$ is trivial.

If $(D, \cdot)$ is a pointed monoid with absorbing element $0_D$, one has the \emph{semiring-algebra} $\widetilde{D} := \mathbb{N}[D]\big\slash\!\langle 0 \equiv 1. 0_D \rangle$, whose elements can be seen as formal finite sums of non-zero elements of $D$, with operations 
\[
\begin{array}{ccc}
\sum n_d. d + \sum m_d. d & = & \sum (n_d+m_d). d\\
\big(\sum n_d. d\big) \big(\sum m_d. d\big) & = &
\underset{d\in D\backslash\{0_D\}}{\sum} \bigg( \underset {a. b = d}{\sum} n_a m_b\bigg) d. 
\end{array}
\]
If $H\subseteq \widetilde{D}\times \widetilde{D}$, we use $\bigbpgenquot{D}{H}$ to denote the ordered blueprint $\bigbpgenquot{(D, \widetilde{D}, =)}{H}$.

There are two canonical functors $(-)^\bullet$ and $(-)^+$ that sends an ordered blueprint to its underlying monoid and ambient semiring, respectively. 

%%%%%%%%%%%%%%%%%%%%%%%%%%%%%%%%%%%%%%%%%%%%%%%%%%%%%%%%%%%%%%%%%%%%%%%%%%%%%

\subsection{Ordered blue modules}

Let $B=(B^\bullet, B^+, \leq_B)$ be an ordered blueprint. An \emph{ordered blue $B$-module}, or simply \emph{$B$-module}, is a triple $M=(M^\bullet,M^+,\leq_M)$ such that:
\begin{enumerate}
 \item $M^+$ is a $B^+$-module;
 \item $\leq_M$ is an additive partial order on $M^+$  such that $(b_1. m_1)\leq_M (b_2. m_2)$ whenever $b_1\leq_B b_2$ and $m_1\leq_M m_2$;
 \item $M^\bullet$ is a pointed $B^\bullet$-subset of $M^+$, where the map $B^\bullet\times M^+\to M^+$ is the restriction of $B^+\times M^+\to M^+$;
 \item $M^\bullet$ generates $M^+$ as a semigroup, \emph{i.e.}, every element of $M^+$ is a finite sum of elements in $M^\bullet$.
\end{enumerate}

We call $M^\bullet$ the \emph{underlying $B^\bullet$-set} and $M^+$ the \emph{ambient $B^+$-module} of $M$. A \emph{morphism of $B$-modules $f:M \to N$} is an order-preserving morphism $f: M^+ \to N^+$ of $B^+$-modules such that $f(M^\bullet)\subseteq N^\bullet$. The category of $B$-modules will be denoted by $B\text{-}\Mod$. 

An \emph{additive $B$-preorder} on $M$ is an additive preorder $\mathfrak{r}$ on the monoid $M^+$ that contains $\leq_M$ and satisfies $b_1. m_1 \leq_\mathfrak{r} b_2. m_2$ whenever $b_1\leq_B b_2$ and $m_1\leq_\mathfrak{r} m_2$. A \emph{congruence} is a $B$-preorder that is symmetric. If $\mathfrak{r}$ is a $B$-preorder on $M$, the set $\mathfrak{c}_\mathfrak{r}:=\{(m, n)|\, m\equiv_\mathfrak{r} n\}$ is a congruence. In this case, we have the \emph{quotient $B$-module} $M/\!\!\!/\mathfrak{r}:=(\overline{M^\bullet}, M^+/\mathfrak{c}_\mathfrak{r}, \overline{\mathfrak{r}})$, where $M^+/\mathfrak{c}_\mathfrak{r}$ is the quotient $B^+$-module, $\overline{M^\bullet}:=\{[m]|\, m\in M^\bullet\}$ and $\overline{\mathfrak{r}}$ is the partial order induced by $\mathfrak{r}$. For a subset $L$ of $M^+\times M^+$, we define the preorder generated by $L$ as
\[
  \langle L \rangle:=\bigcap\{ \text{$B$-preorders on } M^+ \text{ containing } L \text{ and }\leq_M\}.
\]
If $M$ is a $B$-module, let $\mathfrak{g}:=\langle(x,y)\in M^+\times M^+|$ $x\leq_M y \text{ or } y\leq_M x\rangle$. We define the \emph{algebraic hull} of $M$ as the $B$-module $M^{\textup{hull}}:=M/\!\!\!/\mathfrak{g}$. Note that the partial order of $M^{\textup{hull}}$ is trivial.

There are two natural functors $(-)^\bullet$ and $(-)^+$ that send an $B$-module to its underlying $B^\bullet$-set and ambient $B^+$-module, respectively. 

The coproduct of a family $\{M_i=(M_i^\bullet, M_i^+, \leq_i)|\, i \in I\}$ of $B$-modules is given by $M = (M^\bullet, M^+, \leq_M)$, where $M^+=\underset{i\in I}{\bigoplus} M_i^+$, $\leq_M = \langle (m_i)_{i\in I} \leq (n_i)_{i\in I}|\, m_j\leq_j n_j, \, \forall j\in I \rangle$ and $M^\bullet$ is the image of the natural map of $B^\bullet$-sets
\[
\Gamma: \underset{i\in I}{\bigvee} M_i^\bullet \longrightarrow \underset{i\in I}{\bigoplus} M_i^+.
\]

%%%%%%%%%%%%%%%%%%%%%%%%%%%%%%%%%%%%%%%%%%%%%%%%%%%%%%%%%%%%%%%%%%%%%%%%%%%%%

\subsection{Tensor product of $B$-modules}

A \emph{$B$-bilinear map $\varphi: M_1\times M_2 \to N$} is a $B^+$-bilinear map $\varphi: M_1^+\times M_2^+ \to N^+$ such that
\[
\varphi(-, m_2): M_1 \to N \qquad \text{and} \qquad \varphi(m_1, -): M_2 \to N
\]
are morphims of $B$-modules, for all $m_1$ in $M_1^\bullet$ and $m_2$ in $M_2^\bullet$.

One has the natural map of $B^\bullet$-sets 
\[
  \begin{array}{cccc}
   \psi: & M_1^\bullet\otimes_{B^\bullet}M_2^\bullet & \longrightarrow &  M_1^+\otimes_{B^+}M_2^+ \\
    & m_1\otimes m_2& \longmapsto & m_1\otimes m_2.
  \end{array}
 \]
The \emph{tensor product of $M_1$ and $M_2$} is defined as the $B$-module 
\[
M_1\otimes_{B}M_2:=(im\, \psi, M_1^+\otimes_{B^+}M_2^+, =)/\!\!\!/\mathfrak{r},
\]
where $\mathfrak{r}:=\langle x_1\otimes y_1\leq x_2\otimes y_2|\; x_1\leq x_2$ in $M_1^+$ and $y_1\leq y_2 \in M_2^+\rangle$. For $\alpha \in M_1^+\otimes_{B^+}M_2^+$, we denote its class in $(M_1\otimes_{B}M_2)^+$ again by $\alpha$.

Let $\varphi: M_1\times M_2 \to N$ be a $B$-bilinear map. As $\varphi$ is $B^+$-bilinear, there exists a morphism
\[
  \begin{array}{cccc}
   \Tilde{\varphi}: & M_1^+\otimes_{B^+}M_2^+ & \longrightarrow & N^+ \\
    & m_1\otimes m_2& \longmapsto & \varphi(m_1, m_2) 
  \end{array}
 \]
of $B^+$-modules.

By the definition of $B$-bilinearity, $\Tilde{\varphi}(\im \hspace{1mm} \psi)\subseteq N^\bullet$ and $\Tilde{\varphi}(x_1\otimes y_1)\leq \Tilde{\varphi}(x_2\otimes y_2)$ for every relation $x_1\otimes y_1\leq x_2\otimes y_2$ in $M_1^+\otimes_{B^+}M_2^+$. Thus there exists a morphism
\[
  \begin{array}{cccc}
   \overline{\varphi}: & M_1\otimes_{B}M_2 & \longrightarrow & N \\
    & m_1\otimes m_2 & \longmapsto & \varphi(m_1, m_2) 
  \end{array}
 \]
of $B$-modules.

\begin{prop}
The map
\[
  \begin{array}{cccc}
   \Phi: & \Bil _B(M_1\times M_2, N) & \longrightarrow & \Hom      _B(M_1\otimes M_2, N) \\
    & \varphi& \longmapsto & \overline{\varphi} 
  \end{array}
 \]
 is a bijection.
\end{prop}

\begin{proof}
Let $\varphi$ and $\theta$ two $B$-bilinear maps such that $\overline{\varphi}=\overline{\theta}$. Thus, in particular, $\varphi(m_1, m_2)=\overline{\varphi}(m_1\otimes m_1)=\overline{\theta}(m_1\otimes m_1)=\theta(m_1, m_2)$, for all $m_1\in M_1$ and $m_2\in M_2$. With this, we obtain the injectivity of $\Phi$.

Let $\zeta: M_1\otimes_B M_2 \to N$ be a morphism of $B$-modules. Define:
\[
  \begin{array}{cccc}
   \Tilde{\zeta}: & M_1^+\times M_2^+ & \longrightarrow & N^+ \\
    & (m_1, m_2) & \longmapsto & \zeta(m_1\otimes m_2)
  \end{array}
 \]

By the construction of the tensor product and from the fact that $\zeta$ is a morphism, one has that $\Tilde{\zeta}$ is $B$-bilinear. As $\Phi(\Tilde{\zeta})=\zeta$, one has that $\Phi$ is surjective.
\end{proof}

%%%%%%%%%%%%%%%%%%%%%%%%%%%%%%%%%%%%%%%%%%%%%%%%%%%%%%%%%%%%%%%%%%%%%%%%%%%%%%%%%%%%%%%%%%%%%%%%%%%%%%%%%%%%%%%%%%%%%%%%%%%%%%%%%%%%%%%%%%%%%%%%%%%%%%%%%%%%%%%%%%%%%%%%%%%

\section{Exterior algebra}

\subsection{Notations}

For a blueprint $B$, a $B$-module $M$, a natural number $n$ and a set $I$, we will use the following notations: $\{e_i|\, i \in [n]\}\subseteq (B^\bullet)^n$ for the canonical basis; $M^{\otimes n}$ for the $n$-fold tensor product $M\otimes \dotsc \otimes M$; and $M^n$ for the product
\[
\underset{i=1}{\overset{n}{\prod}}M=\bigg(\underset{i=1}{\overset{n}{\prod}}M^\bullet, \underset{i=1}{\overset{n}{\prod}}M^+, \leq_{M^n} \bigg),
\]
where $\leq_{M^n}$ is defined by $(\alpha_1, \dotsc, \alpha _n)\leq_{M^n} (\beta_1, \dotsc, \beta _n)$ if $\alpha_j\leq \beta_j$ for all $j$ in $[n]$.

\begin{df}
Let $\mathbb{F}_1^{\pm}:=(\{0, 1, \epsilon\}, \mathbb{N}\oplus \mathbb{N}. \epsilon, \langle 0\leq 1+\epsilon \rangle)$, where $\epsilon^2=1$.
\end{df}

For the rest of this text, fix an $\mathbb{F}_1^{\pm}$-algebra $B=(B^\bullet, B^+, \leq)$, i.e., an ordered blueprint $B$ equipped with a morphism $\mathbb{F}_1^{\pm} \to B$ or, equivalently, with a distinguished element $\epsilon$ in $B^\bullet$ such that $\epsilon^2=1$ and $0\leq 1+\epsilon$.

%%%%%%%%%%%%%%%%%%%%%%%%%%%%%%%%%%%%%%%%%%%%%%%%%%%%%%%%%%%%%%%%%%%%%%%%%%%%%

\subsection{Construction}

For a $B$-module $M$, let $TM:=\underset{l=0}{\overset{\infty}{\bigoplus}}M^{\otimes l}=(\im\; \Theta, TM^+, \leq)$, where $TM^+:=\underset{l=0}{\overset{\infty}{\bigoplus}}(M^+)^{\otimes l}$ is the $B^+$-tensor algebra and $\Theta: \underset{l=0}{\overset{\infty}{\bigvee}}(M^\bullet)^{\otimes l} \longrightarrow TM^+$ is the natural map of $B^\bullet$-sets.

Note that the usual product of $TM^+$ restricts to a product for $\im \; \Theta$. This operation turns $TM$ into a (typically non-commutative) $B$-algebra.

For $n\in\mathbb{N}$, let 
\[
\tau_n:=\{e_i\otimes e_i\equiv 0|\;i\in[n]\} \cup\{e_i\otimes e_j\equiv \epsilon e_j\otimes e_i|\;i, j \in [n],\, i\neq j\} \subseteq \big((B^n)^{\otimes 2}\big)^+.
\]
If $d\leq n$ and $I=\{i_1, \dotsc, i_d\} \in {[n] \choose d}$ with $i_1<\dotsc < i_d$, define
\[
e_I:=\overline{e_{i_1}\otimes \dotsc \otimes e_{i_d}} \in \bigbpgenquot{T(B^n)}{\tau_n}.
\]
Let
\[
  \begin{array}{cccc}
   \gamma_{d, n}: & {B^+}^{[n] \choose d} & \longrightarrow &  \big(\bigbpgenquot{T(B^n)}{\tau_n}\big)^+\\
   \\
    & (b_I)_{I\in {[n] \choose d}} & \longmapsto & \displaystyle\underset{I\in {[n] \choose d}}{\sum} b_I e_I
  \end{array}
 \]
and define $H_{d, n}:= \gamma_{d, n}\bigg((B^\bullet)^{[n]\choose d}\bigg)$ and $K_{d, n}:= \gamma_{d, n}\bigg((B^\bullet \backslash B^\times)^{[n] \choose d}\bigg)$.

\begin{lemma}
\label{basis-epsilon}
The set $\mathcal{E}:= \{e_I|\, I\subseteq [n]\}$ is a $B^+$-basis of $\big(\bigbpgenquot{T(B^n)}{\tau_n}\big)^+$.
\end{lemma}

\begin{proof}
We begin by noticing that $\mathcal{E}$ clearly is a generator set.

Let $b_I,c_I$ in $B^+$, $I$ subset of $[n]$, such that $\underset{I\subseteq [n]}{\sum} b_Ie_I = \underset{I\subseteq [n]}{\sum} c_Ie_I$. Thus
\begin{equation}
\label{equivalence}
\mathfrak{b} := \underset{i_1<\dotsc<i_d}{\sum} b_{\{i_1, \dotsc, i_d\}}e_{i_1}\otimes \dotsc \otimes e_{i_d}\;
\equiv \underset{i_1<\dotsc<i_d}{\sum} c_{\{i_1, \dotsc, i_d\}}e_{i_1}\otimes \dotsc \otimes e_{i_d} =:\mathfrak{c}.
\end{equation}

By the definition of $\big\langle \tau_n \big\rangle$, there exists two sequences $\mathfrak{b} = x_0, \dotsc, x_m$ and $\mathfrak{c} = y_0, \dotsc, y_w = x_m$ of elements of $T(B^n)^+$ such that $x_{\ell+1} = x_\ell + \alpha$ (resp. $y_{\ell+1} = y_\ell + \alpha$), where $\alpha$ has the form $b. e_a\otimes e_a \otimes e_{z_1}\otimes\dotsc\otimes e_{z_f}$ for some $b$ in $B^+$ and $a, z_1, \dotsc, z_f$ in $[n]$; or $x_{\ell+1} = \rho + \beta_1$ and $x_\ell = \rho + \beta_2$ (resp. $y_{\ell+1} = \rho + \beta_1$ and $y_\ell = \rho + \beta_2$), where $\beta_1$ has the form $b. e_{i_1}\otimes \dotsc \otimes e_{i_f}$ and $\beta_2 = \textup{sign}(\sigma) b. e_{i_{\sigma(1)}} \otimes \dotsc \otimes e_{i_{\sigma(f)}}$, for some $b$ in $B^+$, $\rho$ in $T(B^n)^+$, $i_1, \dotsc, i_f$ in $[n]$, a permutation $\sigma$ and interpreting $\textup{sign}(\sigma)$ as an element of $\F_{1}^\pm$ via the identification of $-1$ with $\epsilon$.

This implies that there are $A_1, A_2$ in $T(B^n)^+$, both of the form $\sum d_{j_1, \dotsc j_t} . e_{j_1}\otimes \dotsc \otimes e_{j_t}$, such that for each set $\{j_1, \dotsc, j_t\}$ one has (at least) two indexes $u,v$ in $[t]$ satisfying $j_u = j_v$; and there are, for each index sets $\{i_1, \dotsc, i_d\}$ present in (\ref{equivalence}), two permutations $\sigma_{\{i_1, \dotsc, i_d\}}$ and $\delta_{\{i_1, \dotsc, i_d\}}$ satisfying
\begin{align}
\label{indexes}
       & \underset{i_1<\dotsc<i_d}{\sum} \textup{sign}(\sigma_{\{i_1, \dotsc, i_d\}}) b_{\{i_1, \dotsc, i_d\}}e_{i_{\sigma_{\{i_1, \dotsc, i_d\}}(1)}}\otimes \dotsc \otimes e_{i_{\sigma_{\{i_1, \dotsc, i_d\}}(d)}} + A_1\\\
\nonumber
= & \underset{i_1<\dotsc<i_d}{\sum} \textup{sign}(\delta_{\{i_1, \dotsc, i_d\}}) c_{\{i_1, \dotsc, i_d\}}e_{i_{\delta_{\{i_1, \dotsc, i_d\}}(1)}}\otimes \dotsc \otimes e_{i_{\delta_{\{i_1, \dotsc, i_d\}}(d)}} + A_2.
\end{align}

Thus, looking at the coefficient of $e_{i_1}\otimes \dotsc \otimes e_{i_d}$ in (\ref{equivalence}) and looking at the coefficient of $e_{i_{\sigma_{\{i_1, \dotsc, i_d\}}(1)}}\otimes \dotsc \otimes e_{i_{\sigma_{\{i_1, \dotsc, i_d\}}(d)}}$ in (\ref{indexes}), we conclude that $b_{\{i_1, \dotsc, i_d\}} = c_{\{i_1, \dotsc, i_d\}}$, for each index set $\{i_1, \dotsc, i_d\}$ present in (\ref{equivalence}).
\end{proof}

Let $S_{d, n}$ be the sub-$B^+$-module of $\big(\bigbpgenquot{T(B^n)}{\tau_n}\big)^+$ generated by $H_{d, n}$. The $B$-module $\bigwedge^d B^n:=\big(H_{d, n},S_{d, n}, \leq\big)$, where $\leq$ is induced from $\bigbpgenquot{T(B^n)}{\tau_n}$, is called \emph{the $d^{th}$ exterior power of $B^n$}.

Let $H_n:=\underset{l=0}{\overset{n}{\bigcup}}H_{l, n}$ and note that it generates $\big(\bigbpgenquot{T(B^n)}{\tau_n}\big)^+$ as a semigroup. The $B$-module $\bigwedge B^n:=(H_n, \big(\bigbpgenquot{T(B^n)}{\tau_n}\big)^+, \leq)$ is called \emph{the exterior algebra of $B^n$}.

The operation $(\overline{\alpha}, \overline{\beta})\mapsto \overline{\alpha\beta}$ defines a product on $\big(\bigwedge B^n)^+$, making it a (typically non-commutative) $B^+$-algebra. For $x, y \in \big(\bigwedge B^n)^+$, we denote by $x\wedge y$ the product of $x$ and $y$. For $I=\{i_1, \dotsc, i_d\}$ with $i_1 < \dotsc < i_d$, we use $e_I$ to denote the element $e_{i_1}\wedge \dotsc \wedge e_{i_d}$.

\begin{rem}
Note that \autoref{basis-epsilon}, in particular, proves \autoref{propositionA}.
\end{rem}

\begin{rem}
For $n\geq 2$, the exterior algebra $\bigwedge B^n$ may not be a $B$-algebra because its underlying pointed set $\big(\bigwedge B^n\big)^\bullet$ could not be multiplicatively closed, as we always have $e_1+e_2$ in $\big(\bigwedge B^n\big)^\bullet$ but, if $1+\epsilon$ is not in $B^\bullet$, $(e_1+e_2) \wedge (e_1+e_2) = (1+\epsilon) \cdot e_{\{1,2\}}$ is not in $\big(\bigwedge B^n\big)^\bullet$.
\end{rem}

\begin{rem}
The difference between the exterior algebra $\bigwedge B^n$ and the $B$-algebra $\bigbpgenquot{T(B^n)}{\tau_n}$ concerns only the underlying pointed $B$-set. This occurs because we need sums of elements of $\big(\bigbpgenquot{T(B^n)}{\tau_n}\big)^\bullet$ to define $B$-Pl\"ucker vectors (cf. \ref{df Plucker vectors}).

But $\bigbpgenquot{T(B^n)}{\tau_n}$ has the following universal property (similar to the universal property of the ring-theoretic exterior algebra): given a (not necessarily commutative) $B$-algebra $A$ and a morphism of $B$-modules $\varphi: B^n\rightarrow A$ satisfying $\varphi(x). \varphi(x) = 0$ and $\varphi(x) . \varphi(y) = \epsilon \varphi(y) . \varphi(x)$ for all $x,y\in B^n$, there exists a unique morphism of $B$-algebras $\Phi: \bigbpgenquot{T(B^n)}{\tau_n} \rightarrow A$ with $\Phi(\overline{x}) = \varphi(x)$ for $x\in B^n$.
\end{rem}

\begin{prop}
The natural maps $\bigwedge^d B^n\to\bigwedge B^n$ induce $\underset{d=0}{\overset{n}{\bigoplus}}\bigwedge^d B^n \simeq \bigwedge B^n$.
\end{prop}

\begin{proof}
By \autoref{basis-epsilon}, one has that $S_{d,n}\bigcap\underset{i\neq d}{\sum} S_{i, n}= \{0\}$. Thus
\[
\big(\bigwedge B^n\big)^+
= \big(\bigbpgenquot{T(B^n)}{\tau_n}\big)^+ 
= \underset{d=0}{\overset{n}{\sum}} S_{d, n} 
= \underset{d=0}{\overset{n}{\bigoplus}} S_{d, n}
= \bigg( \underset{d=0}{\overset{n}{\bigoplus}}\bigwedge^d B^n \bigg)^+
\]
and
\[
\big(\bigwedge B^n\big)^\bullet
= H_n 
= \underset{l=0}{\overset{n}{\bigcup}}H_{l, n} 
= \underset{l=0}{\overset{n}{\bigvee}} H_{l, n}
=  \bigg( \underset{d=0}{\overset{n}{\bigoplus}}\bigwedge^d B^n \bigg)^\bullet.
\]

Let $\leq_1$ be the partial order of $\underset{d=0}{\overset{n}{\bigoplus}}\bigwedge^d B^n$ and $\leq_2$ be the partial order of $\bigwedge B^n$.

As, for each $d$, the partial order of $\bigwedge^d B^n$ is induced from $\leq_2$, one has that $x\leq_1 y$ implies $x\leq_2 y$.

Let $x$ and $y$ in $\big(\bigwedge B^n\big)^+$ such that $x\leq_2 y$. Then there exists $\alpha$ and $\beta$ in $\underset{d=0}{\overset{n}{\bigoplus}} \big((B^n)^+\big)^{\otimes d}$, whose classes in $\big(\bigwedge B^n\big)^+$ are $x$ and $y$, respectively, and such that $\alpha \leq \beta$. Writing $\alpha = \sum \alpha_i$ and $\beta = \sum \beta_i$, with $\alpha_d$ and $\beta_d$ in $\big((B^n)^+\big)^{\otimes d}$, one has that $\alpha_i\leq \beta_i$ for all $i$. Let $x_i$ and $y_i$ the classes of $\alpha_i$ and $\beta_i$ in $\big(\bigwedge B^n\big)^+$, respectively. Thus $x_d$ and $y_d$ are in $S_{n,d}$ and satisfy $x_d\leq_1 y_d$ for all $d$. Therefore $x\leq_1 y$.
\end{proof}

If $R$ is a ring, one has the \emph{monomial ordered blueprint} associated to $R$
\[
R^{\mon}:=\bigbpgenquot{R^\bullet}{1. b\leq \sum 1. a_i|\, b=\sum a_i\in R},
\]
where $R^\bullet$ is the multiplicative underlying monoid of $R$. This is an $\mathbb{F}_1^{\pm}$-algebra, with $\epsilon=1. (-1_R)$.

The next theorem proves the first item of \autoref{teoB} by showing how to recover the usual exterior algebra of rings from our construction of the exterior algebra for $\mathbb{F}_1^{\pm}$-algebras.

\begin{thm}
\label{ring}
Let $R$ be a ring and $\mathcal{B}:=R^{\mon}$. Then $\big(\bigwedge \mathcal{B}^n\big)^{\textup{hull}, +}\simeq \bigwedge R^n$ (the usual exterior algebra) as $R$-algebras.
\end{thm}

\begin{proof}

Note that $\mathcal{B}^{\textup{hull}, +}\simeq R$. Thus one has the isomorphisms of $R$-modules

\[
\big(\bigwedge \mathcal{B}^n\big)^{\textup{hull}, +} \simeq \bigg(\underset{d=0}{\overset{n}{\bigoplus}}\bigwedge^d \mathcal{B}^n\bigg)^{\textup{hull}, +} \simeq \bigg(\underset{d=0}{\overset{n}{\bigoplus}} \mathcal{B}^{ {[n] \choose d} }\bigg)^{\textup{hull}, +} \simeq \underset{d=0}{\overset{n}{\bigoplus}} \,R^{ {[n] \choose d}} \simeq \bigwedge R^n
\]
via
\[
  \begin{array}{ccc}
  
 \big(\bigwedge \mathcal{B}^n\big)^{\textup{hull}, +} & \longrightarrow & \bigwedge R^n\\
  
  [e_{i_1}\wedge \dotsc \wedge e_{i_d}] & \longmapsto & e_{i_1}\wedge \dotsc \wedge e_{i_d},
  \end{array}
 \]
where $[x]$ denotes the class of $x\in \big(\bigwedge \mathcal{B}^n\big)^+$ in $\big(\bigwedge \mathcal{B}^n\big)^{\textup{hull}, +}$.

As this map is a morphism of $R$-algebras, we have the result.
\end{proof}

For an idempotent semifield $S$, let
\[
\Omega_S:=\bigg\{1. b\leq \underset{i=1}{\overset{m}{\sum}}1. a_i\bigg|\, \underset{i=1}{\overset{m}{\sum}}a_i=b+\underset{i\neq k}{\underset{i=1}{\overset{m}{\sum}}}a_i \mbox{ in } S,\, \forall k\in [m]\bigg\}
\]
and define the \emph{monomial ordered blueprint} associated to $S$ as $S^{\mon}:=\bigbpgenquot{S^\bullet}{\Omega_S}$. This is an $\mathbb{F}_1^{\pm}$-algebra, with $\epsilon=1. 1_S$ and the construction above extends to a functor $(-)^{\mon}: \IdempSFields \to \OBlpr$.

For an ordered blueprint $C$, let $C^{\textup{idem}}:=\bigbpgenquot{C}{1\equiv 1+1}$ be the \emph{idempotent ordered blueprint} associated to $C$. This name comes from the fact that $C^{\textup{idem}, +}$ is always an idempotent semiring.

The following definition is due to Jeffrey and Noah Giansiracusa (cf. \cite[Definition 3.1.2]{Giansiracusa18}).

\begin{df}
Let $S$ be an idempotent semifield. For $n\in \mathbb{N}$, let
\[\Psi_n:=\{e_i\otimes e_j \equiv e_j \otimes e_i|\, i,j\in[n]\} \cup \{e_i\otimes e_i \equiv 0|\, i\in[n]\} \subseteq (S^n)^{\otimes 2}.
\]
The tropical Grassmann algebra of $S^n$ is the graded $S$-algebra
\[
\bigwedge S^n 
:= Sym\; S^n \big\slash\!\big\langle\,e_i^2 \,\equiv\, 0|\, i\in[n]\,\big\rangle 
\simeq T(S^n)\big\slash\!\big\langle\,\Psi_n\,\big\rangle.
\]

The $d^{th}$-homogeneous direct summand of $\bigwedge S^n$, denoted by $\bigwedge^d S^n$, is called the $d^{th}$ \textit{tropical wedge power of} $S^n$.
\end{df}

The next theorem proves the second item of \autoref{teoB} by showing how to recover the tropical Grassmann algebra from our construction of the exterior algebra and exterior powers for $\mathbb{F}_1^{\pm}$-algebras.

\begin{thm}
\label{semifield}
Let $S$ be an idempotent semifield and $\mathcal{S}:=S^{\mon}$. 
Then $\big(\bigwedge^d \mathcal{S}^n\big)^{\textup{idem},+}\simeq \bigwedge^d S^n$, for all $d$, and $\big(\bigwedge \mathcal{S}^n\big)^{\textup{idem},+}\simeq \bigwedge S^n$, where $\bigwedge^d S^n$ and $\bigwedge S^n$ denotes the Giansiracusa $d^{th}$ tropical wedge power and tropical Grassmann algebra of $S^n$, respectively.
\end{thm}

\begin{proof}
To begin with, we prove that $\mathcal{S}^{\textup{idem},+}\simeq S$ as semirings. We will denote the class of $\sum n_i . s_i \in \mathcal{S}^{+}$ in $\mathcal{S}^{\textup{idem}, +}$ by $[\sum n_i . s_i]$. One has the natural map
\[
\begin{array}{cccc}
\varphi: & \mathcal{S}^{\textup{idem},+} & \longrightarrow & S \\
 & \left[\sum n_i . s_i \right]& \longmapsto & \sum n_i  s_i,
\end{array}
\]
which is a surjective morphism of semirings. Note that $\bigg[\underset{i=1}{\overset{m}{\sum}} m_i . s_i\bigg]=\bigg[\underset{i=1}{\overset{m}{\sum}}1. s_i\bigg]$ in $\mathcal{S}^{\textup{idem},+}$, for all positive integers $m_i$ and $s_i$ in $S$.

Let $a$, $b$ and $c:=a+b$ be elements of $S$. Note that $c + b = c = c + a$ and $c+c = a + c$. Thus, in $\mathcal{S}^+$, one has $1. c\leq 1. a + 1. b$ and $1 . a \leq 2 . c$. Analogously, $1 . b \leq 2 . c$. Therefore, $1 . a + 1 . b \leq 4 . c$ in $\mathcal{S}^+$, which implies $1. c\leq 1. a + 1. b\leq 4 . c = 1. c$. Thus, $1. c = 1. a + 1. b$ in $\mathcal{S}^{\textup{idem},+}$. As a consequence, we obtain that $\varphi$ is injective.

Next we observe that $\big(\bigwedge^d \mathcal{S}^n\big)^+ 
  \simeq \bigg(\mathcal{S}^{[n] \choose d}\bigg)^+ 
  \simeq \big(\mathcal{S}^+\big)^{[n] \choose d} \mbox{ as } \mathcal{S}^+\mbox{-modules}$, which implies
\[
\begin{array}{c}
\big(\bigwedge^d \mathcal{S}^n\big)^{\textup{idem}, +}  \simeq \bigg(\mathcal{S}^{[n] \choose d}\bigg)^{\textup{idem}, +}
  \simeq \big(\mathcal{S}^{\textup{idem}, +}\big)^{[n] \choose d}
  \simeq S^{[n] \choose d} \simeq \bigwedge^d S^n
  \end{array}
 \]
as $S$-modules.

Thus $\big(\bigwedge \mathcal{S}^n\big)^{\textup{idem}, +}\simeq \underset{d=0}{\overset{n}{\bigoplus}}\big(\bigwedge^d \mathcal{S}^n\big)^+ \simeq \underset{d=0}{\overset{n}{\bigoplus}}\big(\bigwedge^d S^n\big) \simeq \bigwedge S^n$, via
\[
  \begin{array}{ccc}
  \big(\bigwedge \mathcal{S}^n\big)^{\textup{idem}, +} & \longrightarrow & \bigwedge S^n\\
  
  [e_I] & \longmapsto & e_I.
  \end{array}
  \]
As this map is a morphism of $S$-algebras, we have the result.
\end{proof}

\section{Matroids}

In this section we will focus on matroids over blueprints and prove \autoref{teoC}.

\begin{df}
\label{df B-matroid}
A \textit{Grassmann-Pl\"ucker function of rank $d$ on $[n]$ with coefficients in $B$} is a function $\Delta: {[n] \choose d} \to B^\bullet$ such that $\Delta(J)$ is in $B^\times$ for some $J$ and satisfies the Pl\"ucker relations
\[
\displaystyle\underset{i_k\notin X}{\sum} \epsilon^k \Delta(X\cup \{i_k\})\Delta(Y\backslash \{i_k\})\geq 0
\]
whenever $X\in {[n] \choose d-1}$ and $Y=\{i_1, \dotsc, i_{d+1}\}\in {[n] \choose d+1}$, with $i_1 < \dotsc < i_{d+1}$.

We say that two Grassmann-Pl\"ucker functions $\Delta$ and $\Delta'$ are equivalent, and write $\Delta\sim\Delta'$, if there is some $a$ in $B^\times$ such that $\Delta(I)=a. \Delta'(I)$ for all $I$. A \textit{$B$-matroid (of rank $d$ on $[n]$)} is an equivalence class of Grassmann-Pl\"ucker functions (of rank $d$ on $[n]$) with coefficients in $B$ (cf. \cite[Definition 5.1]{Baker-Lorscheid21}).
\end{df}

\begin{rem}
Besides other applications, the order in blueprint theory is used to describe relations analogous to algebraic equations. In particular, the \textit{positive cone} $\{b\in B^+|\, b\geq 0\}$ in Baker and Lorscheid's theory of $\F^\pm_1$-algebras plays the same role as of the null set in Baker and Bowler's theory of tracts, generalizing previous concepts of zero elements and sets (cf. \cite[Section 1.2.4]{Baker-Lorscheid21} and \cite[Section 2.3]{Baker-Bowler19}).
\end{rem}

For $X\in {[n] \choose d-1}$ and $Y=\{i_1, \dotsc, i_{d+1}\} \in {[n] \choose d+1}$ with $i_1 < \dotsc < i_{d+1}$, there exists a (unique) morphism of $B$-modules $\varphi_{X,Y}: \bigwedge^d B^n \otimes \bigwedge^d B^n \to B$ that sends $e_I\otimes e_J$ to $\epsilon^k$ if there is $k \in [d+1]$ such that $I=X\cup\{i_k\}$ and $J=Y\backslash \{i_k\}$, and sends $e_I\otimes e_J$ to $0$ otherwise.

\begin{df}
\label{df Plucker vectors}

A rank $d$ $B$-Pl\"ucker vector is a $v\in H_{d, n}\backslash K_{d, n}$ satisfying $\varphi_{X, Y}(v\otimes v)\geq 0$, i.e., a $v=\displaystyle\underset{I\in {[n] \choose d}}{\sum} v_I e_I \in H_{d, n}$ such that there exists a $J$ with $v_J\in B^\times$ and, for all $X\in {[n] \choose d-1}$ and $Y\in {[n] \choose d+1}$, one has
\[
\underset{i_k\notin X}{\sum}\epsilon^k v_{X\cup \{i_k\}} v_{Y\backslash \{i_k\}} \geq 0,
\]
where $Y=\{i_1, \dotsc, i_{d+1}\}$ with $i_1 < \dotsc < i_{d+1}$.

Note that $B^\times$ acts on the set of $B$-Pl\"ucker vectors, and more generally on $H_{d,n}$, by usual multiplication. So we can take the equivalence class of a $B$-Pl\"ucker vector in $H_{d,n}/B^{\times}$.
\end{df}

\begin{thm}
\label{teo vectors}
There is a bijection
\[
   \begin{array}{ccc}
   \{\text{rank d $B$-Pl\"ucker vectors}\} & \longrightarrow & \{\text{Grassmann-Pl\"ucker functions }  \Delta: {[n] \choose d} \rightarrow B^\bullet \}\\
   v=\displaystyle\underset{I\in {[n] \choose d}}{\sum} v_I e_I & \longmapsto & \Delta: I \longmapsto v_I
  \end{array}
\]
\end{thm}

\begin{proof}
It follows from \autoref{basis-epsilon} and the definitions above.
\end{proof}

The next corollaries prove \autoref{teoC}.

\begin{cor}
Let $K$ be a field and $B:= K^\mon$. Then there exists a bijection between $B$-matroids of rank $d$ on $[n]$ and $K$-matroids of rank $d$ on $[n]$.
\end{cor}

\begin{proof}
Due to the description of $K$-matroids via Grassmann-Pl\"ucker functions (cf. \cite[p. 841]{Baker-Bowler19}), the bijection is reached by composing the Grassmann-Pl\"ucker functions presented in \autoref{teo vectors} with the functor $( - )^{\hull, +}$ and taking equivalence classes on both sides.
\end{proof}

The following is a reformulation of Definition 4.1.1 from \cite{Giansiracusa18}.

\begin{df}
\label{df-S-matroid}
Let $S$ be an indempotent semifield. A \textit{rank $d$ tropical $S$-Pl\"ucker vector} is a nonzero $v = \sum v_I e_I \in \bigwedge^d S^n$ satisfying, for all subsets $A\in {[n] \choose d+1}$ and $X \in {[n] \choose d-1}$,
\[
\underset{i\in A\backslash X}{\sum} v_{a-i} v_{X+i} = 
\underset{i\in A\backslash X, \;i\neq p}{\sum} v_{a-i} v_{X+i}
\]
in $S$, for all $p\in A\backslash X$.
\end{df}

\begin{cor}
Let $S$ be an idempotent semifield and $B:=S^\mon$. Then there exists a bijection between rank $d$ $B$-Pl\"ucker vectors and rank $d$ tropical $S$-Pl\"ucker vectors.    
\end{cor}

\begin{proof}
Note that, due to the description of the relations $\Omega_S$ on the construction of $S^\mon$, one has that $v=\displaystyle\underset{I\in {[n] \choose d}}{\sum} v_I e_I$, with $v_I \in B^\bullet$, is a $B$-Pl\"ucker vector if and only if $\widetilde{v}=\displaystyle\underset{I\in {[n] \choose d}}{\sum} v_I e_I$, with $v_I \in S$, is a tropical Pl\"ucker vector.
\end{proof}

\begin{cor}
There is a bijection between rank $d$ $B$-matroids and classes of rank $d$ $B$-Pl\"ucker vectors.
\end{cor}

\begin{proof}
We just need to take the equivalence class (by the action of $B^\times$) on both sides of the bijection presented in \autoref{teo vectors}.
\end{proof}

%%%%%%%%%%%%%%%%%%%%%%%%%%%%%%%%%%%%%%%%%%%%%%%%%%%%%%%%%%%%%%%%%%%%%%%%%%%%%%%%%%%%%%%%%%%%%%%%%%%%%%%%%%%%%%%%%%%%%%%%%%%%%%%%%%%%%%%%%%%%%%%%%%%%%%%%%%%%%%%%%%%%%%%%%%%

\bibliographystyle{alpha}
\bibliography{exterior}

\end{document}